\documentclass[10pt, reqno]{amsart}

\usepackage{amssymb}
\usepackage{enumitem}
\usepackage{hyperref}
\usepackage[noabbrev]{cleveref}
\usepackage{mathtools}

\usepackage[backend=biber, doi=false, giveninits=true, isbn=false, sortcites=true,
style=numeric-comp, url=false, sorting=nyt]{biblatex} 
\renewbibmacro{in:}{}  
\allowdisplaybreaks

\numberwithin{equation}{section}

    \renewcommand{\epsilon}{\varepsilon}
    \renewcommand{\phi}{\varphi}

    \renewcommand{\tilde}{\widetilde}
    \renewcommand{\hat}{\widehat}

    \newcommand{\defeq}{\vcentcolon=}
    \newcommand{\eqdef}{=\vcentcolon}

    \newcommand{\set}[1]{\{ #1 \}}

    \newcommand{\R}{\mathbb{R}}

    \newcommand{\abs}[1]{\lvert #1 \rvert}
    \newcommand{\bigabs}[1]{\bigl\lvert #1 \bigr\rvert}
    \newcommand{\Abs}[1]{\left\lvert #1 \right\rvert}
    
    \newcommand{\biginner}[2]{\bigl\langle #1, #2 \bigr\rangle}
    \newcommand{\Inner}[2]{\left\langle #1, #2 \right\rangle}
    \newcommand{\norm}[1]{\lVert #1 \rVert}
    \newcommand{\Norm}[1]{\left\lVert #1 \right\rVert}

    \newcommand{\les}{\lesssim}

    \newcommand{\grad}{\nabla}
    \newcommand{\lap}{\Delta}

    \newcommand{\bigmes}{\bigabs}

    \newcommand{\bigO}[2][]{\mathcal{O}_{#1}(#2)}
    \newcommand{\BigO}[2][]{\mathcal{O}_{#1}{\left(#2\right)}}
    \newcommand{\lilo}[2][]{o_{#1}(#2)}

    \newcommand{\wprop}{\mathcal{U}}
    \newcommand{\wop}[1]{W_{\mkern-5mu \vphantom{\tilde{F}} #1}}
    \newcommand{\sop}[1]{S_{\mkern-3mu \vphantom{\tilde{F}} #1}}
    \newcommand{\dpt}[1]{D_{\mkern-3mu \vphantom{\tilde{F}} #1}}

\newtheorem{theorem}{Theorem}[section]
\newtheorem{proposition}[theorem]{Proposition}
\newtheorem{lemma}[theorem]{Lemma}
\newtheorem{corollary}[theorem]{Corollary}

\theoremstyle{definition}
\newtheorem{definition}[theorem]{Definition}

\theoremstyle{remark}

\begin{document}

\title[Deconvolutional determination of the nonlinearity]{Deconvolutional determination of
the nonlinearity in a semilinear wave equation}

\author[N.~Hu]{Nicholas Hu}
\address{Department of Mathematics, UCLA}
\email{njhu@math.ucla.edu}

\author[R.~Killip]{Rowan Killip}
\address{Department of Mathematics, UCLA}
\email{killip@math.ucla.edu}

\author[M.~Vișan]{Monica Vișan}
\address{Department of Mathematics, UCLA}
\email{visan@math.ucla.edu}

\maketitle

\begin{abstract} 
    We demonstrate that in three space dimensions, the scattering behaviour of semilinear
    wave equations with quintic-type nonlinearities uniquely determines the nonlinearity.
    The nonlinearity is permitted to depend on both space and time.
\end{abstract}

\section{Introduction}

We consider the semilinear wave equation
\begin{equation} \label{nlw}
    \begin{cases}
        (\partial_{tt} - \lap_x) u(t, x) = F(t, x, u(t, x)), & (t, x) \in \R \times \R^3; \\
        u(0,{}\cdot{}) = u_0; \\
        \partial_t u(0,{}\cdot{}) = u_1.
    \end{cases}
\end{equation}
Under mild assumptions on the nonlinearity $F : \R \times \R^3 \times \R \to \R$, we show
that this equation admits a small-data scattering theory and that the scattering operator
\emph{determines} the nonlinearity.  The specific class of nonlinearities we consider is
given in Definition~\ref{D:admissible} and may be regarded as a generalization of the
energy-critical case.  The main inspiration for the problem we study is the paper
\cite{SaBarreto} of S\'a Barreto, Uhlmann, and Wang.  Our methods, however, are more
strongly influenced by Killip, Murphy, and Vișan \cite{Killip}.

The requirements that we impose on the nonlinearity are as follows:

\begin{definition}[Admissible nonlinearity] \label{D:admissible}
    A measurable function $F : \R \times \R^3 \times \R \to \R$ will be called
    \emph{admissible} for \cref{nlw} if
    \begin{enumerate}[label=(\roman*)]
        \item
            $F(t, x, 0) = 0$ for all $t, x$;
        \item
            $\abs{F(t, x, u) - F(t, x, v)} \les (\abs{u}^4 + \abs{v}^4) \abs{u - v}$ for all
            $u, v$ uniformly in $t, x$; and \item $F(t, x, -u) = -F(t, x, u)$ for all $t,
            x$.
    \end{enumerate}
\end{definition}

If $F(t,x,u)=\pm |u|^4u$, the resulting equation is the defocusing/focusing (depending on
the sign of the nonlinearity) energy-critical wave equation.  This name reflects the fact
that in this case, the equation enjoys a scaling symmetry $$ u(t,x)\mapsto u^\lambda(t,x) =
\lambda ^{\frac12} u\bigl(\lambda t, \lambda x\bigr) \qquad \text{for $\lambda>0$} $$ that
preserves the energy of solutions $$ E(u)= \int_{\R^3} \tfrac12 |\nabla u(t,x)|^2 +\tfrac12
|\partial_t u(t,x)|^2 \pm \tfrac 16|u(t,x)|^6\, dx.  $$ Accordingly, we will be studying
equation \eqref{nlw} with initial data $(u_0,u_1)$ in the energy space $\dot H^1(\R^3)\times
L^2(\R^3)$.

\begin{definition}[Solution] \label{D:solution}
    A function $u : \R \times \R^3 \to \R$ is said to be a \emph{strong 
    global solution} of \cref{nlw} if $(u, \partial_t u) \in C^0_t \dot{H}^1_x (\R \times
    \R^3) \times C^0_t L^2_x (\R \times \R^3)$, $u \in L^5_t L^{10}_x (K \times \R^3)$ for
    all compact sets $K \subseteq \R$, and $u$ satisfies the Duhamel formula
    \[
        \begin{bmatrix} u(t) \\ \partial_t u(t) \end{bmatrix} = 
        \wprop(t) \begin{bmatrix} u_0 \\ u_1 \end{bmatrix} + 
        \int_0^t \wprop(t-s) \begin{bmatrix} 0 \\ F(s) \end{bmatrix} \, ds.
    \]
Here $\wprop$ denotes the propagator for the linear wave equation, that is,
    {\renewcommand{\arraystretch}{1.5}
        \[
            \wprop(t) \defeq 
            \begin{bmatrix} 
                \cos(t\abs{\grad}) & \dfrac{\sin(t\abs{\grad})}{\abs{\grad}} \\
                -\abs{\grad} \sin(t\abs{\grad}) & \cos(t\abs{\grad}) 
            \end{bmatrix}.
        \]
    }%
    Here and in what follows we abbreviate $u(t, {}\cdot{})$ as $u(t)$ and $F(t, {}\cdot{},
    u(t))$ as $F(t)$.
\end{definition}

For admissible nonlinearities, equation \eqref{nlw} admits a small-data global
well-posedness and scattering theory.

\begin{theorem}[Small-data scattering] \label{T:scattering}
    Let $F$ be an admissible nonlinearity for \cref{nlw}.
    Then there exists an $\eta > 0$ such that \cref{nlw} has a unique global solution $u$
    satisfying
    \begin{equation} \label{soln-estimate}
        \norm{(u, \partial_t u)}_{L^\infty_t \dot{H}^1_x \times L^\infty_t L^2_x} +
        \norm{u}_{L^5_t L^{10}_x} 
        \les \norm{(u_0, u_1)}_{\dot{H}^1 \times L^2}
    \end{equation}
    whenever $(u_0, u_1) \in B_\eta$, where
    \[
        B_\eta \defeq 
        \set{(u_0, u_1) \in \dot{H}^1(\R^3) \times L^2(\R^3) : 
        \norm{(u_0, u_1)}_{\dot{H}^1 \times L^2} < \eta}.
    \]

    This solution scatters in $\dot{H}^1(\R^3) \times L^2(\R^3)$ as $t \to \pm \infty$, 
    meaning that there exist
    \textup(necessarily unique\textup) asymptotic states 
    $(u^\pm_0, u^\pm_1) \in \dot{H}^1(\R^3) \times L^2(\R^3)$ for which
    \begin{equation} \label{scattering}
        \Norm{\begin{bmatrix} u(t) \\ \partial_t u(t) \end{bmatrix} -
            \wprop(t) \begin{bmatrix} u^\pm_0 \\ u^\pm_1 \end{bmatrix}}_{\dot{H}^1 \times
        L^2}
        \to 0
        \quad
        \text{as $t \to \pm \infty$}.
    \end{equation}

    In addition, for all $(u^-_0, u^-_1) \in B_\eta$, there exists a unique global
    solution $u$ to \cref{nlw} and a unique asymptotic state $(u^+_0, u^+_1) \in
    \dot{H}^1(\R^3) \times L^2(\R^3)$ so that both limits in \labelcref{scattering} hold.
\end{theorem}

The map $(u_0, u_1) \mapsto (u^+_0, u^+_1)$ defined implicitly by \Cref{T:scattering} on the
open ball $B_\eta \subseteq \dot{H}^1(\R^3) \times L^2(\R^3)$ is the inverse of what is
often called the \emph{forward wave operator}; in this paper, we will refer to it simply as
the \emph{wave operator} and we will denote it by $\wop{F}$.  The map $(u^-_0, u^-_1)
\mapsto (u^+_0, u^+_1)$ is the \emph{scattering operator} and will be denoted $\sop{F}$.
Our principal result is that either operator determines the nonlinearity completely.

Our hypotheses on the nonlinearity $F$ do not demand any continuity in $t$ or $x$.  Avoiding
such a restriction is important for us as we wish to allow nonlinearities of the form
$1_\Omega (x) u^5$, which model a nonlinear medium (whose shape we wish to determine)
surrounded by vacuum.

Without a continuity requirement, complete determination of the nonlinearity means
determination at (Lebesgue) almost every spacetime point.  We can be very precise about the
spacetime points at which we determine the nonlinearity:

\begin{definition}[Determinable point]
    Suppose that $F$ is an admissible nonlinearity for \cref{nlw}. A point $(t, x) \in \R
    \times \R^3$ will be called \emph{determinable} for $F$ if it is a Lebesgue point of
    $F({}\cdot{}, {}\cdot{}, u)$ for every rational $u$. The set of all such points will be
    denoted $\dpt{F}$.
\end{definition}

For each fixed $u$, the map $(t,x)\mapsto F(t,x,u)$ is bounded and measurable and so almost
every point is a Lebesgue point.  The countability of the rational numbers then guarantees
that almost every spacetime point is determinable.  

\begin{theorem} \label{T:main}
    Suppose that $F$ and $\tilde{F}$ are admissible nonlinearities for \cref{nlw} and that
    $B_\eta$ and $B_{\tilde{\eta}}$ are corresponding balls given by \Cref{T:scattering}.
    If $\wop{F}$ and $\wop{\tilde{F}}$, or $\sop{F}$ and $\sop{\tilde{F}}$, agree on $B_\eta
    \cap B_{\tilde{\eta}}$ \textup(that is, the smaller of the two balls\textup), then $F(t,
    x, {}\cdot{}) = \tilde{F}(t, x, {}\cdot{})$ for all $(t, x) \in \dpt{F} \cap
    \dpt{\tilde{F}}$. 
\end{theorem}

The question of whether the nonlinearity in a dispersive PDE is determined by its scattering
behaviour has been extensively studied \cite{SaBarreto, Carles, Morawetz, Pausader, Sasaki,
Sasaki2, Sasaki3, Sasaki4, Watanabe, Watanabe2, Watanabe3, Watanabe4, Weder, Weder2, Weder3,
Weder4, Weder5, Weder6}. Usually, rather strong assumptions are imposed on the nonlinearity
in order to obtain a positive answer.

In contrast, Killip, Murphy, and Vișan's deconvolution-based approach \cite{Killip} enabled
them to determine power-type nonlinearities in a semilinear Schr\"odinger equation with only
moderate growth restrictions on the nonlinearities.  Their approach is flexible and
technically simple, as demonstrated by its subsequent application to the determination of
coefficients \cite{Murphy, Hogan} and inhomogeneities \cite{Chen, Chen2} of nonlinear
Schr\"odinger equations.

In this paper, we revisit the setting considered by S\'a Barreto, Uhlmann, and Wang
\cite{SaBarreto}, who determined nonlinearities of the form $F = F(u)$ in
\cref{nlw} under the following assumptions:
\begin{enumerate}[label=(\roman*)]
    \item
        $F(u) = h(u) u$ for some even function $h$ satisfying $\abs{h(u)} \approx
        \abs{u}^4$ for all $u$;

    \item
        $F'(u) u \sim F(u)$ as $u \to 0$ and as $u \to \pm \infty$;

    \item
        $u \mapsto \int_0^u F(v) \, dv$ is convex;

    \item
        $\abs{F^{(j)}(u)} \les \abs{u}^{5-j}$ for each $0 \leq j \leq 5$; and

    \item
        $F^{(4)}(u) = 0$ if and only if $u = 0$.
\end{enumerate}

By adapting the deconvolution technique of \cite{Killip} to the setting of the wave
equation, we will prove that even more general nonlinearities of the form $F = F(t, x, u)$
can be determined under the weaker conditions of \Cref{D:admissible}.

Let us now turn to an overview of the paper, the method of \cite{Killip}, and the principal
challenges to be overcome in applying it in the wave equation setting.

Our first task is to establish the existence, uniqueness, and long-time behaviour of
solutions to \eqref{nlw} for small initial data and for admissible nonlinearities.  This is
Theorem~\ref{T:scattering}, which we prove in Section~\ref{S:scattering}.

Following \cite{Killip}, our approach to identifying the nonlinearity is through the
small-data asymptotics of the scattering and wave operators.  These are presented in
Corollary~\ref{C:asymp}, which gives a precise estimate on the difference between the full
operators and what is known as their \emph{Born approximation}.

Under the Born approximation, the scattering/wave operators capture the spacetime integral
of $u(t,x)F(t,x,u(t,x))$, where $u(t,x)$ is a solution of the \emph{linear} wave equation.
This evidently represents a substantial `blurring' of the nonlinearity across different
values of $t$, $x$, and $u$.  If the nonlinearity did not depend on $t$ and $x$, then this
blurring would take the form of a convolution (over the multiplication group).  By switching
to exponential variables, this then would become a convolution in the traditional sense.  In
this way, the question of identifying the nonlinearity reduces to a deconvolution problem.
As we will discuss in Section~\ref{S:deconvolution}, the uniqueness criterion for such
deconvolution problems is the well-known $L^1$ Tauberian theorem of Wiener; see
Theorem~\ref{T:wiener}.   

To overcome the dependence of the nonlinearity on space and time, we will employ a solution
of the linear wave equation that concentrates tightly at a single point in spacetime (while
also remaining small in scaling-critical norms).  As noted earlier, we do not assume that
the nonlinearity is continuous in $t$ or $x$; consequently, there are some subtleties to be
overcome in localizing the nonlinearity to a single spacetime point.  This is the role of
Lemma~\ref{L:leb-diff}.  With this hurdle overcome, the uniqueness question is reduced to
the deconvolution problem presented in Proposition~\ref{P:conv}.

We now arrive at the crux of the matter: we need to find solutions to the linear wave
equation that lead to a deconvolution problem that can actually be solved.   Concretely, we
must find a linear solution whose distribution function we can compute sufficiently
explicitly that we will be able to verify the hypotheses of Wiener's Tauberian theorem.  The
distribution function for the solution we choose is computed in Lemma~\ref{L:measures}.
Although we are unable to compute the resulting Fourier transform precisely, we are
nonetheless able to verify that it is nonvanishing (see Proposition~\ref{P:w-ft}) and
consequently to apply the Tauberian theorem.

\hypersetup{bookmarksdepth=-2}  
\subsection*{Acknowledgements} R.K. was supported by NSF grant DMS-2154022; M.V. was
supported by NSF grant DMS-2054194.
\hypersetup{bookmarksdepth}

\subsection{Notation}
Throughout this paper, we employ the standard notation $A \les B$ to indicate that $A \leq
CB$ for some constant $C > 0$; if $A \les B$ and $B \les A$, we write $A \approx B$.
Occasionally, we adjoin subscripts to this notation to indicate dependence of the constant
$C$ on other parameters; for instance, we write $A \les_{\alpha, \beta} B$ when $A \leq
CB$ for some constant $C > 0$ depending on $\alpha, \beta$.

\section{Small-data scattering} \label{S:scattering}

We begin by establishing the small-data scattering theory described in \Cref{T:scattering}.
This relies on a standard contraction mapping argument using Strichartz estimates. 

\begin{theorem}[Strichartz estimates, \cite{Pecher, Strichartz, Keel}]
    If $u : \R \times \R^3 \to \R$ is a global solution of \cref{nlw}, then
    \[
        \norm{(u, \partial_t u)}_{L^\infty_t \dot{H}^1_x \times L^\infty_t L^2_x} +
        \norm{u}_{L^5_t L^{10}_x} \les
        \norm{(u_0, u_1)}_{\dot{H}^1 \times L^2} + \norm{F(t, x, u(t, x))}_{L^1_t L^2_x}.
    \]
\end{theorem}

The contraction mapping argument constructs the solution from the Duhamel formula
\begin{equation} \label{duhamel}
    \begin{bmatrix} u(t) \\ \partial_t u(t) \end{bmatrix} = 
    \wprop(t) \begin{bmatrix} u_0 \\ u_1 \end{bmatrix} + 
    \int_0^t \wprop(t-s) \begin{bmatrix} 0 \\ F(s) \end{bmatrix} \, ds.
\end{equation}
Similarly, the solution with prescribed asymptotic state $(u_0^-, u_1^-)$ as $t\to -\infty$
is constructed from the formula
\begin{equation} \label{duhamel-scattering}
    \begin{bmatrix} u(t) \\ \partial_t u(t) \end{bmatrix} = 
    \wprop(t) \begin{bmatrix} u^-_0 \\ u^-_1 \end{bmatrix} + 
    \int_{-\infty}^t \wprop(t-s) \begin{bmatrix} 0 \\ F(s) \end{bmatrix} \, ds.
\end{equation}

\begin{proof}[Proof of \Cref{T:scattering}]
    Let
    \begin{align*}
        X &\defeq \big\{u : \R \times \R^3 \to \R : \\
          &\mathrel{\hphantom{\defeq}}\hphantom{\{}
              (u, \partial_t u) \in C^0_t \dot{H}^1_x(\R \times \R^3) \times C^0_t L^2_x(\R
              \times \R^3),\,
              u \in L^5_t L^{10}_x(\R \times \R^3), \\
          &\mathrel{\hphantom{\defeq}}\hphantom{\{}
              \norm{(u, \partial_t u)}_{L^\infty_t \dot{H}^1_x \times L^\infty_t L^2_x} +
              \norm{u}_{L^5_t L^{10}_x} \leq
              2C \norm{(u_0, u_1)}_{\dot{H}^1 \times L^2}\big\},
   \end{align*}
where $C$ is the implicit constant in the Strichartz estimates. Equipping $X$ with the
metric
$$
d(u, v) \defeq \norm{(u, \partial_t u) - (v, \partial_t v)}_{L^\infty_t \dot{H}^1_x \times
        L^\infty_t L^2_x} +\norm{u-v}_{L^5_t L^{10}_x}\,,
$$        
we obtain a nonempty complete metric space $(X, d)$.

    For $u \in X$, we then define
    \[
        (\Phi(u))(t) \defeq 
        \cos(t\abs{\grad}) u_0 + 
        \frac{\sin(t\abs{\grad})}{\abs{\grad}} u_1 + 
        \int_0^t \frac{\sin((t-s)\abs{\grad})}{\abs{\grad}} F(s) \, ds
    \]
    so that
        \begin{equation} \label{duhamel-phi}
        \begin{bmatrix} (\Phi(u))(t) \\ (\partial_t \Phi(u))(t) \end{bmatrix} = 
        \wprop(t) \begin{bmatrix} u_0 \\ u_1 \end{bmatrix} + 
        \int_0^t \wprop(t-s) \begin{bmatrix} 0 \\ F(s) \end{bmatrix} \, ds.
    \end{equation}
    To construct the solution of \cref{nlw}, we will show that $\Phi$ is a contraction on
    $(X, d)$ whenever $(u_0, u_1) \in B_\eta$ and $\eta$ is sufficiently small.  The
    solution sought will then be the fixed point of $\Phi$ whose existence and uniqueness
    are guaranteed by the Banach fixed point theorem.

    We first verify that $\Phi$ maps $X$ into itself.  Let $C_F$ be a constant such that
    $\abs{F(t, x, u)} \leq C_F \abs{u}^5$ for all $(t, x) \in \R \times \R^3$.  If $u \in
    X$, then by the Strichartz estimates, we have
    \begin{align*}
        &\norm{(\Phi(u), \partial_t \Phi(u))}_{L^\infty_t \dot{H}^1_x \times L^\infty_t
        L^2_x} + \norm{\Phi(u)}_{L^5_t L^{10}_x} \\
        &\quad\leq C(\norm{(u_0, u_1)}_{\dot{H}^1 \times L^2} + 
        \norm{F(t, x, u(t, x))}_{L^1_t L^2_x}) \\
        &\quad\leq C(\norm{(u_0, u_1)}_{\dot{H}^1 \times L^2} + C_F \norm{u}_{L^5_t
        L^{10}_x}^5) \\
        &\quad\leq C[1 + C_F (2C \eta)^4 (2C)] 
        \norm{(u_0, u_1)}_{\dot{H}^1 \times L^2} \\
        &\quad\leq 2C \norm{(u_0, u_1)}_{\dot{H}^1 \times L^2},
    \end{align*}
    provided that $\eta$ is sufficiently small.

    To show that $(\Phi(u))(t)$ and $(\partial_t \Phi(u))(t)$ are also \emph{continuous} in
    $t$, fix a $t_0 \in \R$ and consider, without loss of generality, the case when $t \geq
    t_0$.  The first term on the right-hand side of formula \labelcref{duhamel-phi}
    converges to $\wprop(t_0)(u_0, u_1)$ in $\dot{H}^1 \times L^2$ as $t \to t_0$ since
    $\wprop(t)$ is strongly continuous in $t$. As for the second term, we observe that
    \begin{align*}
        &\Norm{
            \int_{0}^{t}
            \wprop(-s)
        \begin{bmatrix} 0 \\ F(s) \end{bmatrix} \, ds -
            \int_{0}^{t_0}
            \wprop(-s)
        \begin{bmatrix} 0 \\ F(s) \end{bmatrix} \, ds}_{\dot{H}^1 \times L^2} \\
        &\quad\leq
        \Norm{
            \int_{t_0}^{t}
        \frac{\sin(-s\abs{\grad})}{\abs{\grad}} F(s) \, ds}_{\dot{H}^1} +
        \Norm{
            \int_{t_0}^{t}
        \cos(-s\abs{\grad}) F(s) \, ds}_{L^2} \\
        &\quad\les
        \int_{t_0}^{t} \norm{F(s)}_{L^2} \, ds \\
        &\quad \les\norm{u}_{L^5_t L^{10}_x ([t_0, t] \times \R^3)}^5 \to 0 \quad \text{as
        $t\to t_0$},
    \end{align*}
    by the dominated convergence theorem. Consequently, the second term converges to $\,
    \wprop(t_0) \int_0^{t_0} \wprop(-s) (0, F(s)) \, ds$ in $\dot{H}^1 \times L^2$ as $t \to
    t_0$ since $\wprop(t)$ is strongly continuous and uniformly bounded in $t$.  Altogether,
    this shows that $\Phi(u) \in X$ as required.

    Now if $u, v \in X$, the Strichartz estimates also yield
    \begin{align*}
        d(\Phi(u), \Phi(v))
        &\les \norm{F(t, x, u(t, x)) - F(t, x, v(t, x))}_{L^1_t L^2_x} \\
        &\les \norm{(\abs{u}^4 + \abs{v}^4) \abs{u-v}}_{L^1_t L^2_x} \\
        &\les \bigl(\norm{u}_{L^5_t L^{10}_x}^4 + \norm{v}_{L^5_t L^{10}_x}^4\bigr) 
        \norm{u-v}_{L^5_t L^{10}_x} \\
        &\les [(2C\eta)^4 + (2C\eta)^4]\,
        d(u, v),
    \end{align*}
    which shows that $\Phi$ is a contraction for sufficiently small $\eta$.

    Next, we prove that the solution $u$ scatters in $\dot{H}^1 \times L^2$. As $\wprop(t)$
    is unitary on $\dot{H}^1 \times L^2$, this amounts to showing that the
    functions $\wprop^{-1}(t) (u(t), \partial_t u(t))$ converge in $\dot{H}^1 \times L^2$ as
    $t \to \pm \infty$. By time reversal symmetry, it suffices to consider $t \to
    +\infty$. For $t_2 \geq t_1 \geq T$,
    \begin{align*}
        &\Norm{\,\wprop^{-1}(t_2) 
            \begin{bmatrix} u(t) \\ \partial_t u(t) \end{bmatrix} -
            \wprop^{-1}(t_1) 
        \begin{bmatrix} u(t) \\ \partial_t u(t) \end{bmatrix}}_{\dot{H}^1 \times L^2} \\
        &\quad=
        \Norm{
            \int_{0}^{t_2}
            \wprop(-s)
        \begin{bmatrix} 0 \\ F(s) \end{bmatrix} \, ds -
            \int_{0}^{t_1}
            \wprop(-s)
        \begin{bmatrix} 0 \\ F(s) \end{bmatrix} \, ds
        }_{\dot{H}^1 \times L^2} \\
        &\quad\les
        \norm{u}_{L^5_t L^{10}_x ([t_1, t_2] \times \R^3)}^5\to 0 \quad\text{as $T\to
        \infty$},
    \end{align*}
    by the dominated convergence theorem. We conclude that $\set{\wprop^{-1}(t)(u(t),
    \partial_t u(t))}$ is Cauchy in $\dot{H}^1 \times L^2$ as $t \to \infty$ and therefore
    convergent.

    This completes the construction of the wave operator. The construction of the scattering
    operator, using \labelcref{duhamel-scattering} in place of \labelcref{duhamel}, is
    entirely analogous.
\end{proof}

We note that the foregoing argument shows that the wave operator is given by
\begin{align}
    \wop{F} \left(\begin{bmatrix} u_0 \\ u_1 \end{bmatrix}\right) &=
    \begin{bmatrix} u_0 \\ u_1 \end{bmatrix} + 
    \int_0^{\infty} \wprop(-t) \begin{bmatrix} 0 \\ F(t) \end{bmatrix} \, dt,
    \label{duhamel-wop}
\intertext{where $u$ is the solution of \cref{nlw} with initial data $(u_0, u_1)$.
Similarly, the scattering operator is given by}
    \sop{F} \left(\begin{bmatrix} u^-_0 \\ u^-_1 \end{bmatrix}\right) &=
    \begin{bmatrix} u^-_0 \\ u^-_1 \end{bmatrix} + 
    \int_{-\infty}^{\infty} \wprop(-t) \begin{bmatrix} 0 \\ F(t) \end{bmatrix} \, dt,
    \label{duhamel-sop}
\end{align}
where $u$ is the solution of \cref{nlw} that scatters to $(u^-_0, u^-_1)$ as $t \to
-\infty$.

\begin{corollary}[Small-data asymptotics for the wave and scattering operators]
\label{C:asymp}
    Suppose that $F$ is an admissible nonlinearity for \cref{nlw} and that $B_\eta$ is a
    corresponding ball given by \Cref{T:scattering}. 
    If $u_\mathrm{lin}$ denotes the solution of the linear wave equation with initial data
    $(u_0, u_1) \in B_\eta$, then \textup(in $\dot{H}^1 \times L^2$\textup) we have
    \begin{align}
        \wop{F} \left(\begin{bmatrix} u_0 \\ u_1 \end{bmatrix}\right) &=
        \begin{bmatrix} u_0 \\ u_1 \end{bmatrix} + 
        \int_0^{\infty} \wprop(-t) \begin{bmatrix} 0 \\ F_\mathrm{lin}(t) \end{bmatrix}
        \, dt +
        \BigO{\Norm{\begin{bmatrix} u_0 \\ u_1 \end{bmatrix}}_{\dot{H}^1 \times L^2}^9}. 
        \label{wop-asymp}
    \intertext{Similarly, given $(u^-_0, u^-_1) \in B_\eta$, let $u$ be the solution of 
    \cref{nlw} that scatters to $(u^-_0, u^-_1)$ as $t \to -\infty$.
    If $u_\mathrm{lin}$ denotes the solution of the linear wave equation with 
    initial data $(u_0, u_1) \defeq (u(0), \partial_t u(0)) \in B_\eta$, then}
        \sop{F} \left(\begin{bmatrix} u_0^- \\ u_1^- \end{bmatrix}\right) &=
        \begin{bmatrix} u_0^- \\ u_1^- \end{bmatrix} + 
        \int_{-\infty}^{\infty} \wprop(-t) 
        \begin{bmatrix} 0 \\ F_\mathrm{lin}(t) \end{bmatrix} \, dt +
        \BigO{\Norm{\begin{bmatrix} u_0 \\ u_1 \end{bmatrix}}_{\dot{H}^1 \times L^2}^9}.
        \label{sop-asymp}
    \end{align}
    Here $F_\mathrm{lin}(t)$ is an abbreviation for $F(t, {}\cdot{}, u_\mathrm{lin}(t))$.
\end{corollary}

\begin{proof}
    We will derive the asymptotic expansion \labelcref{sop-asymp} 
    from formula \labelcref{duhamel-sop} for the scattering
    operator; the derivation of \labelcref{wop-asymp} 
    from formula \labelcref{duhamel-wop}
    for the wave operator is similar.

    Comparing \labelcref{duhamel-sop} with \labelcref{sop-asymp},
    we see that the latter follows from
    \[
        \Norm{\int_{-\infty}^\infty \wprop(-t) \begin{bmatrix} 0 \\ F(t) -
        F_\mathrm{lin}(t) \end{bmatrix} \, dt}_{\dot{H}^1 \times L^2}
        \les
        \Norm{\begin{bmatrix} u_0 \\ u_1 \end{bmatrix}}_{\dot{H}^1 \times L^2}^9,
    \]
    which we will prove by duality.
    To this end, fix some $(v_0, v_1) \in \dot{H}^1 \times L^2$ and 
    let $v_\mathrm{lin}$ denote
    the solution of the linear wave equation with initial data $(v_0, v_1)$. Then
    \begin{align*}
        &\Inner{\int_{-\infty}^\infty \wprop(-t) \begin{bmatrix} 0 \\ F(t) -
        F_\mathrm{lin}(t) \end{bmatrix} \, dt}{\begin{bmatrix} v_0 \\ v_1
        \end{bmatrix}}_{\dot{H}^1 \times L^2} \\
        &\quad=
        \int_{-\infty}^\infty \Inner{\begin{bmatrix} 0 \\ F(t) -
        F_\mathrm{lin}(t) \end{bmatrix}}{\,\wprop(t) \begin{bmatrix} v_0 \\ v_1
        \end{bmatrix}}_{\dot{H}^1 \times L^2} \, dt \\
        &\quad=
        \int_{-\infty}^\infty \Inner{\begin{bmatrix} 0 \\ F(t) -
        F_\mathrm{lin}(t) \end{bmatrix}}{\begin{bmatrix} v_\mathrm{lin}(t) \\
        \partial_t v_\mathrm{lin}(t) \end{bmatrix}}_{\dot{H}^1 \times L^2} \, dt \\
        &\quad=
        \int_{-\infty}^\infty \biginner{F(t) - 
        F_\mathrm{lin}(t)}{\partial_t v_\mathrm{lin}(t)}_{L^2} \, dt.
    \end{align*}
    As a result, it will suffice to show that
    \begin{equation} \label{sop-asymp-duality}
        \Abs{\int_{-\infty}^\infty \biginner{F(t) - 
        F_\mathrm{lin}(t)}{\partial_t v_\mathrm{lin}(t)}_{L^2} \, dt}
        \les
        \Norm{\begin{bmatrix} u_0 \\ u_1 \end{bmatrix}}_{\dot{H}^1 \times L^2}^9
        \Norm{\begin{bmatrix} v_0 \\ v_1 \end{bmatrix}}_{\dot{H}^1 \times L^2}.
    \end{equation}
    To estimate this integral, we first employ H\"older's inequality to deduce that
    \begin{align*}
        &\Abs{\int_{-\infty}^\infty 
        \biginner{F(t) - F_\mathrm{lin}(t)}{\partial_t v_{\mathrm{lin}}(t)}_{L^2} \, dt} \\
        &\quad\leq \norm{F(t, x, u(t, x)) - F(t, x, u_\mathrm{lin}(t, x))}_{L^{1}_t L^{2}_x}
        \cdot
        \norm{\partial_t v_{\mathrm{lin}}}_{L^{\infty}_t L^{2}_x} \\
        &\quad\les \bigl(\norm{u}^4_{L^5_t L^{10}_x} + \norm{u_\mathrm{lin}}^4_{L^5_t
        L^{10}_x}\bigr)
        \norm{u - u_\mathrm{lin}}_{L^{5}_t L^{10}_x}
        \cdot
        \norm{\partial_t v_\mathrm{lin}}_{L^{\infty}_t L^{2}_x}\,.
        \stepcounter{equation} \tag{\theequation} \label{sop-asymp-holder}
    \end{align*}
    By \labelcref{soln-estimate} and the Strichartz estimates, we have
    \begin{align*}
        \norm{u}_{L^5_t L^{10}_x}^4
        &\les \norm{(u_0, u_1)}_{\dot{H}^1 \times L^2}^4\,,
        \\
        \norm{u_\mathrm{lin}}_{L^5_t L^{10}_x}^4
        &\les \norm{(u_0, u_1)}_{\dot{H}^1 \times L^2}^4\,,
        \\
        \norm{u - u_\mathrm{lin}}_{L^5_t L^{10}_x} 
        &\les \norm{F(t, x, u(t, x))}_{L^1_t L^2_x}
        \les \norm{u}^5_{L^5_t L^{10}_x}
        \les \norm{(u_0, u_1)}_{\dot{H}^1 \times L^2}^5\,,
        \\
        \norm{\partial_t v_\mathrm{lin}}_{L^\infty_t L^2_x} 
        &\les \norm{(v_0, v_1)}_{\dot{H}^1 \times L^2}\,.
    \end{align*}
    Inserting these estimates into \labelcref{sop-asymp-holder} 
    yields \labelcref{sop-asymp-duality}, completing the proof of the corollary.
\end{proof}

\section{Reduction to a convolution equation} \label{S:reduction}

The next step is the reduction of the proof of \Cref{T:main} to the consideration of a
convolution equation.  As in \cite{Killip}, the central idea is to exploit the Born
approximation for well-chosen solutions of the linear wave equation.  Indeed, the principal
obstacle to be overcome in implementing that strategy is to find solutions of the linear
wave equation with the key properties we need.  Most fundamentally, we need solutions for
which we are not only able to compute the distribution function (i.e., the measure of
spacetime superlevel sets), but can also prove that the Fourier transform of a certain
function $w$ connected with it does not vanish.

Our solutions will be built from the radially symmetric solution
\[
    u_\mathrm{lin}(t, r) \defeq \frac{f(r-t) - f(r+t)}{r}
\]
of the linear wave equation $(\partial_{tt} - \lap_x) u(t, x) = 0$ on $\R \times \R^3$, 
where $r \defeq \abs{x}$ and $f(s) \defeq \max \set{1-\abs{s}, 0}$.
This solution arises from the initial data
\begin{equation} \label{initial-data}
    \begin{split}
    u_0(x) &\defeq
    u_\mathrm{lin}(0, \abs{x}) 
    = 0 \in \dot{H}^1(\R^3), \\
    u_1(x) &\defeq
    \partial_t u_\mathrm{lin}(0, \abs{x}) 
    =
    \begin{cases*}
        \frac{2}{\abs{x}} & if $0 < \abs{x} \leq 1$, \\
        0 & if $\abs{x} > 1$
    \end{cases*}
    \in L^2(\R^3).
    \end{split}
\end{equation}
In addition, $u_\mathrm{lin}(t, x) \geq 0$ for $t > 0$ and $u_\mathrm{lin}(t, x)$ is odd in
$t$.

The next lemma gives a formula for the distribution function of $u_\mathrm{lin}$.  The
function $w$ connected with this solution is presented in \eqref{weight}.  The nonvanishing
of the Fourier transform of $w$ will be demonstrated in Proposition~\ref{P:w-ft}. 

\begin{lemma} \label{L:measures}
    For $\lambda > 0$, let
    \begin{align*}
        m(\lambda) 
        &\defeq \bigmes{\set{(t, x) \in (0, \infty) \times \R^3 : 
        u_\mathrm{lin}(t, x) > \lambda}}.
    \end{align*}
    Then
    \[
        m(\lambda) 
        = \frac{4\pi}{3\vphantom{\lambda^3}} 
        \left(\frac{1}{2\lambda^3} - \frac{2}{(\lambda+2)^3}\right)
        1_{(0, 2)}(\lambda).
    \]
\end{lemma}

\begin{proof}
    For $t, \lambda > 0$, let
    \[
        m(t; \lambda) \defeq \bigmes{\set{x \in \R^3 : 
        u_\mathrm{lin}(t, x) > \lambda}}
    \]
    so that
    \begin{align}\label{m}
        m(\lambda) = \int_0^\infty m(t; \lambda) \, dt.
    \end{align}

    We will evaluate this integral by analyzing $u_\mathrm{lin}$ on the spacetime regions 
    $0 < t < \frac{1}{2}$, $\frac{1}{2} < t < 1$, and $t > 1$.

    On the region $0 < t < \frac{1}{2}$, we have
    \[
        u_\mathrm{lin}(t, r) = 
        \begin{cases*}
            2 & if $0 < r < t$, \\
            \frac{2t}{r} & if $t \leq r < 1-t$, \\
            \frac{1-r+t}{r} & if $1-t \leq r < 1+t$, \\
            0 & otherwise.
        \end{cases*}
    \]
    Hence, for $0 < t < \frac{1}{2}$,
    \[
        m(t; \lambda) =
        \frac{4\pi}{3} \cdot
        \begin{cases*}
            (\frac{1+t}{1+\lambda})^3 & if $0 < \lambda < \frac{2t}{1-t}$, \\
            (\frac{2t}{\lambda})^3 & if $\frac{2t}{1-t} \leq \lambda < 2$, \\
            0 & otherwise.
        \end{cases*}
    \]
    Therefore, the contribution of this region to the right-hand side of \eqref{m} is
    \begin{align}
        \int_0^\frac{1}{2}\!\! m(t; \lambda) \, dt &= \!\int_0^\frac{1}{2} \!\frac{4\pi}{3}
        \left(\frac{1+t}{1+\lambda}\right)^3\!\!
        1_{\set{0 < \lambda < \frac{2t}{1-t}}}(t) \, dt
        + \int_0^\frac{1}{2} \!\frac{4\pi}{3}
        \left(\frac{2t}{\lambda}\right)^3 \!\!
        1_{\set{\frac{2t}{1-t} \leq \lambda < 2}}(t) \, dt \notag\\
        &= 
        \left[
            \int_\frac{\lambda}{\lambda+2}^\frac{1}{2} \frac{4\pi}{3}
            \left(\frac{1+t}{1+\lambda}\right)^3 \, dt
        \right]
        1_{(0, 2)}(\lambda)
        + 
        \left[
            \int_0^\frac{\lambda}{\lambda+2} \frac{4\pi}{3}
            \left(\frac{2t}{\lambda}\right)^3 \, dt
        \right]
        1_{(0, 2)}(\lambda) \notag\\
        &=
        \frac{4\pi}{3\vphantom{()^3}}
        \left(\frac{81}{64(\lambda+1)^3} - \frac{4(\lambda+1)}{(\lambda+2)^4}\right)
        1_{(0, 2)}(\lambda) +
        \frac{4\pi}{3\vphantom{()^3}} \cdot\frac{2\lambda}{(\lambda+2)^4}\,
        1_{(0, 2)}(\lambda) \notag\\
        &= \frac{4\pi}{3\vphantom{()^3}}
        \left(\frac{81}{64(\lambda+1)^3} - \frac{2}{(\lambda+2)^3}\right) 1_{(0,
        2)}(\lambda). \label{int1}
    \end{align}

    On the region $\frac{1}{2} < t < 1$, we have
    \[
        u_\mathrm{lin}(t, r) = 
        \begin{cases*}
            2 & if $0 < r < 1-t$, \\
            \frac{1+r-t}{r} & if $1-t \leq r < t$, \\
            \frac{1-r+t}{r} & if $t \leq r < 1+t$, \\
            0 & otherwise.
        \end{cases*}
    \]
    Hence, for $\frac{1}{2} < t < 1$,
    \[
        m(t; \lambda) =
        \frac{4\pi}{3} \cdot
        \begin{cases*}
            (\frac{1+t}{1+\lambda})^3 & if $0 < \lambda < \frac{1}{t}$, \\
            (\frac{1-t}{\lambda-1})^3 & if $\frac{1}{t} \leq \lambda < 2$, \\
            0 & otherwise.
        \end{cases*}
    \]
    Therefore, the contribution of this region to the right-hand side of \eqref{m} is
    \begin{align*}
        \int_\frac{1}{2}^1 m(t; & \lambda) \, dt = \int_\frac{1}{2}^1 \frac{4\pi}{3}
        \left(\frac{1+t}{1+\lambda}\right)^3 \!\!1_{\set{0 < \lambda < \frac{1}{t}}}(t) \,
        dt
        + \int_\frac{1}{2}^1 \frac{4\pi}{3}
        \left(\frac{1-t}{\lambda-1}\right)^3 \!\!
        1_{\set{\frac{1}{t} \leq \lambda < 2}}(t) \, dt \\
        &= \left[ \int_\frac{1}{2}^1 \frac{4\pi}{3}
        \left(\frac{1+t}{1+\lambda}\right)^3 \, dt \right] 1_{(0, 1]}(\lambda)
        + \left[ \int_\frac{1}{2}^\frac{1}{\lambda} \frac{4\pi}{3}
        \left(\frac{1+t}{1+\lambda}\right)^3 \, dt \right] 1_{(1, 2)}(\lambda) \\
        &\hphantom{{}={}}
        + \left[ \int_\frac{1}{\lambda}^1 \frac{4\pi}{3}
        \left(\frac{1-t}{\lambda-1}\right)^3 \, dt \right] 1_{(1, 2)}(\lambda) \\
        &= \frac{4\pi}{3\vphantom{()^3}}\cdot
        \frac{175}{64(\lambda+1)^3} \,1_{(0, 1]}(\lambda)
        + \frac{4\pi}{3\vphantom{()^3}}
        \left(\frac{\lambda+1}{4\lambda^4} - \frac{81}{64(\lambda+1)^3}\right) 1_{(1,
        2)}(\lambda) \\
        &\hphantom{{}={}}
        + \frac{4\pi}{3\vphantom{\lambda^4}}
        \cdot\frac{\lambda-1}{4\lambda^4}\, 1_{(1, 2)}(\lambda) \\
        &= \frac{4\pi}{3\vphantom{()^3}}
        \cdot\frac{175}{64(\lambda+1)^3} \,1_{(0, 1]}(\lambda)
        + \frac{4\pi}{3\vphantom{()^3}}
        \left(\frac{1}{2\lambda^3} - \frac{81}{64(\lambda+1)^3}\right) 1_{(1,
        2)}(\lambda).
        \addtocounter{equation}{1}\tag{\theequation}\label{int2}
    \end{align*}

    On the region $t > 1$, we have
    \[
        u_\mathrm{lin}(t, r) = 
        \begin{cases*}
            \frac{1+r-t}{r} & if $t-1 \leq r < t$, \\
            \frac{1-r+t}{r} & if $t \leq r < t+1$, \\
            0 & otherwise.
        \end{cases*}
    \]
    Hence, for $t > 1$,
    \[
        m(t; \lambda) =
        \frac{4\pi}{3} \cdot
        \begin{cases*}
            (\frac{t+1}{\lambda+1})^3 - (\frac{t-1}{1-\lambda})^3 & if $0 < \lambda <
            \frac{1}{t}$, \\
            0 & otherwise.
        \end{cases*}
    \]
    Therefore, the contribution of this region to the right-hand side of \eqref{m} is
    \begin{align*}
        \int_1^\infty m(t; \lambda) \, dt
        &= \int_1^\infty \frac{4\pi}{3}
        \left[\left(\frac{t+1}{\lambda+1}\right)^3 - 
        \left(\frac{t-1}{1-\lambda}\right)^3\right] 
        1_{\set{0 < \lambda < \frac{1}{t}}}(t) \, dt \\
        &= \left\{ \int_1^\frac{1}{\lambda} \frac{4\pi}{3}
        \left[\left(\frac{t+1}{\lambda+1}\right)^3 - 
        \left(\frac{t-1}{1-\lambda}\right)^3\right] \, dt \right\} 
        1_{(0, 1)}(\lambda) \\
        &= \frac{4\pi}{3\vphantom{()^3}}
        \left(\frac{1}{2\lambda^3} - \frac{4}{(\lambda+1)^3}\right) 1_{(0, 1)}(\lambda).
        \addtocounter{equation}{1}\tag{\theequation}\label{int3}
    \end{align*}

    Finally, combining \eqref{m} with  \labelcref{int1}, \labelcref{int2}, and
    \labelcref{int3} completes the proof of the lemma.
\end{proof}

To continue, we generate further solutions of the linear wave equation using the scaling
symmetry.  Specifically, for positive parameters $\alpha$ and $\epsilon$, the rescaled
function $u_\mathrm{lin}^{\alpha, \epsilon}$ defined as
\begin{align*}
    u_\mathrm{lin}^{\alpha, \epsilon}(t, x) 
    &\defeq \alpha u_\mathrm{lin}((\alpha/\epsilon)^2 t, (\alpha/\epsilon)^2 x)
\end{align*}
solves the linear wave equation with initial data
\begin{align*}
    u_0^{\alpha, \epsilon}(x) 
    &\defeq 
    u_\mathrm{lin}^{\alpha, \epsilon}(0, x) 
    =
    \alpha u_0((\alpha/\epsilon)^2 x)
    = 0,
    \\
    u_1^{\alpha, \epsilon}(x)
    &\defeq
    \partial_t u_\mathrm{lin}^{\alpha, \epsilon}(0, x) 
    = (\alpha/\epsilon)^2 \alpha u_1((\alpha/\epsilon)^2 x).
\end{align*}
Under this rescaling, we have $\norm{u_0^{\alpha, \epsilon}}_{\dot{H}^1} = \epsilon
\norm{u_0}_{\dot{H}^1}$ and $\norm{u_1^{\alpha, \epsilon}}_{L^2} = \epsilon
\norm{u_1}_{L^2}$, so from \labelcref{initial-data} we compute that
\[
    \norm{(u_0^{\alpha, \epsilon}, u_1^{\alpha, \epsilon})}_{\dot{H}^1 \times L^2}^2
    = \epsilon^2 \norm{(u_0, u_1)}_{\dot{H}^1 \times L^2}^2
    = 16 \pi \epsilon^2.
\]
In particular, if $F$ is an admissible nonlinearity for \cref{nlw} and $B_\eta$ is a
corresponding ball given by \Cref{T:scattering}, then 
$(u_0^{\alpha, \epsilon}, u_1^{\alpha, \epsilon}) \in B_\eta$ 
for all sufficiently small $\epsilon$.

We will also rely on the observation that $u_\mathrm{lin}(t, x) = \partial_t
v_\mathrm{lin}(t, x)$, where $v_\mathrm{lin}$ is itself a radially symmetric solution of the
linear wave equation on $\R \times \R^3$ with initial data
\begin{align*}
    v_0(x) &\defeq v_\mathrm{lin}(0, x) =
    \begin{cases*}
        \abs{x}-2 & if $0 < \abs{x} \leq 1$, \\
        -\frac{1}{\abs{x}} & if $\abs{x} > 1$
    \end{cases*}
    \in \dot{H}^1(\R^3), \\
    v_1(x) & \defeq
    \partial_t v_\mathrm{lin}(0, x) = 0 \in L^2(\R^3).
\end{align*}
Thus, $u_\mathrm{lin}^{\alpha, \epsilon}(t, x) = \partial_t v^{\alpha,
\epsilon}_\mathrm{lin}(t, x)$, where the rescaled function $v_\mathrm{lin}^{\alpha,
\epsilon}$ defined as
\begin{align*}
    v_\mathrm{lin}^{\alpha, \epsilon}(t, x) 
    &\defeq (\alpha/\epsilon)^{-2} 
    \alpha v_\mathrm{lin}((\alpha/\epsilon)^2 t, (\alpha/\epsilon)^2 x)
\intertext{solves the linear wave equation with initial data}
    v_0^{\alpha, \epsilon}(x)
    &\defeq 
    v_\mathrm{lin}^{\alpha, \epsilon}(0, x) 
    =
    (\alpha/\epsilon)^{-2} \alpha v_0((\alpha/\epsilon)^2 x),
    \\
    v_1^{\alpha, \epsilon}(x)
    &\defeq 
    \partial_t v_\mathrm{lin}^{\alpha, \epsilon}(0, x) 
    =
    \alpha v_1((\alpha/\epsilon)^2 x)
    = 0.
\end{align*}
Under this rescaling, we have
\[
    \norm{(v_0^{\alpha, \epsilon}, v_1^{\alpha, \epsilon})}_{\dot{H}^1 \times L^2}^2
    = (\alpha/\epsilon)^{-6} \alpha^2 \norm{(v_0, v_1)}_{\dot{H}^1 \times L^2}^2\,
    = 16 \pi \epsilon^6 / 3\alpha^4.
\]

\begin{proposition}[Reduction to a convolution equation] \label{P:conv}
    Suppose that $F$ and $\tilde{F}$ are admissible nonlinearities for \cref{nlw}.
    For $(t_0, x_0) \in \dpt{F} \cap \dpt{\tilde{F}}$ and $\tau \in \R$, define
    \begin{align*}
        H(\tau; t_0, x_0) &\defeq 
        e^{-4\tau} \frac{\partial F}{\partial u}(t_0, x_0, e^\tau)  + 
        e^{-5\tau} F(t_0, x_0, e^\tau) , \\
        \tilde{H}(\tau; t_0, x_0) &\defeq 
        e^{-4\tau} \frac{\partial \tilde{F}}{\partial u}(t_0, x_0, e^\tau) + 
        e^{-5\tau} \tilde{F}(t_0, x_0, e^\tau) .
    \end{align*}
    Then $H$ and $\tilde{H}$ are bounded
    and, under the hypotheses of \Cref{T:main}, we have
    \begin{equation} \label{conv-eqn}
        H * w = \tilde{H} * w,
    \end{equation}
    where
    \begin{equation} \label{weight}
        w(\tau) \defeq \left(e^{-3\tau} - \frac{4e^{-6\tau}}{(e^{-\tau} +
        1)^3}\right) 1_{(0, \infty)}(\tau).
    \end{equation}
\end{proposition}

The proof of this proposition relies on the following result, which shows that in the Born
approximation described in Corollary~\ref{C:asymp}, we may replace $F(t,x,u)$ by $F(t_0,x_0,
u)$ up to acceptable errors.

\begin{lemma}\label{L:leb-diff}
    Suppose that $F$ is an admissible nonlinearity for \cref{nlw}.
    Then for all $(t_0, x_0) \in \dpt{F}$, we have 
    \begin{align*}
        &\int_{-\infty}^\infty 
        \biginner{F(t, x, u_\mathrm{lin}^{\alpha, \epsilon}(t-t_0, x-x_0))}{
        u_\mathrm{lin}^{\alpha, \epsilon}(t-t_0, x-x_0)}_{L^2_x} \, dt \\
        &\quad=
        \int_{-\infty}^\infty 
        \biginner{F(t_0, x_0, u_\mathrm{lin}^{\alpha, \epsilon}(t-t_0, x-x_0))}{
        u_\mathrm{lin}^{\alpha, \epsilon}(t-t_0, x-x_0)}_{L^2_x} \, dt 
        + \lilo[\alpha]{\epsilon^8}
    \end{align*}
as $\epsilon\to 0$.
\end{lemma}

We postpone the proof of Lemma~\ref{L:leb-diff} until after we have completed that of
Proposition~\ref{P:conv}.

\begin{proof}[Proof of \Cref{P:conv}]
    We only consider the case where the scattering operators agree, as the wave operators
    can be treated similarly. By time and space translation symmetry, it suffices to treat
    the case $(t_0, x_0) = (0,0)$.

    Let $G(t, x, u) \defeq F(t, x, u) u$ so that
    \begin{align*}
        \int_{-\infty}^\infty
        \biginner{F(0, 0, u_\mathrm{lin}(t, x))}{u_\mathrm{lin}(t, x)}_{L^2_x} \, dt
        &=
        \int_{-\infty}^\infty \int_{\R^3} 
        G(0, 0, u_\mathrm{lin}(t, x)) \, dx \, dt.
    \intertext{%
    By the fundamental theorem of calculus, Fubini's theorem, and \Cref{L:measures},
    }
        \int_{-\infty}^\infty \int_{\R^3} 
        G(0, 0, u_\mathrm{lin}(t, x)) \, dx \, dt
        &=
        2
        \int_0^\infty
        \frac{\partial G}{\partial u}(0, 0, \lambda)
        m(\lambda) \, d\lambda.
    \intertext{Hence,}
        \int_{-\infty}^\infty
        \biginner{F(0, 0, u_\mathrm{lin}^{\alpha, \epsilon}(t, x))}{
        u_\mathrm{lin}^{\alpha, \epsilon}(t, x)}_{L^2_x} \, dt
        &=
        2
        (\alpha/\epsilon)^{-8}
        \int_0^\infty
        \frac{\partial G}{\partial u}(0, 0, \lambda)
        m(\lambda/\alpha) \, d\lambda.
    \end{align*}
    Performing the change of variables $\lambda \eqdef e^{\tau}$, we obtain
    \begin{align*}
        &\int_{-\infty}^\infty
        \biginner{F(0, 0, u_\mathrm{lin}^{\alpha, \epsilon}(t, x))}{
        u_\mathrm{lin}^{\alpha, \epsilon}(t, x)}_{L^2_x} \, dt
        \\
        &\quad=
        \frac{2\epsilon^8}{\alpha^8}
        \int_{-\infty}^\infty
        \frac{\partial G}{\partial u}(0, 0, e^\tau) e^\tau
        m(e^{\tau - \log \alpha}) \, d\tau \\
        &\quad=
        \frac{2\epsilon^8}{\alpha^8}
        \int_{-\infty}^\infty
        H(\tau) e^{6\tau}
        m(e^{\tau - \log \alpha}) 
        \, d\tau \\
        &\quad=
        \frac{2\epsilon^8}{\alpha^8} \cdot \frac{(2\alpha)^6 \pi}{12}
        \int_{-\infty}^\infty
        H(\tau)
        \cdot
        \frac{12}{\pi\vphantom{\alpha^8}}
        e^{-6(\log 2\alpha - \tau)} m(e^{\tau - \log \alpha}) \, d\tau \\
        &\quad=
        \frac{32 \pi \epsilon^8}{3 \alpha^2}
        \int_{-\infty}^\infty
        H(\tau) w(\log 2\alpha - \tau) \, d\tau \\
        &\quad=
        \frac{32 \pi \epsilon^8}{3 \alpha^2}
        (H * w)(\log 2\alpha),
        \addtocounter{equation}{1}\tag{\theequation} \label{F-to-H}
    \end{align*}
    where $w(\tau) = \frac{12}{\pi} e^{-6\tau} m(e^{-(\tau-\log 2)})$ is as given by
    \labelcref{weight}.

    On the other hand, if $F_\mathrm{lin}^{\alpha, \epsilon}(t) \defeq F(t, {}\cdot{},
    u_\mathrm{lin}^{\alpha, \epsilon}(t))$, then
    \begin{align*}
        \int_{-\infty}^\infty
        \biginner{F_\mathrm{lin}^{\alpha, \epsilon}(t)}{
        u_\mathrm{lin}^{\alpha, \epsilon}(t)}_{L^2} \, dt
        &=
        \int_{-\infty}^\infty
        \biginner{F_\mathrm{lin}^{\alpha, \epsilon}(t)}{\partial_t v_\mathrm{lin}^{\alpha,
        \epsilon}(t)}_{L^2} \, dt \\
        &=
        \int_{-\infty}^\infty
        \Inner{\begin{bmatrix} 0 \\ F_\mathrm{lin}^{\alpha, \epsilon}(t)
        \end{bmatrix}}{\,\wprop(t) \begin{bmatrix} v_0^{\alpha, \epsilon} \\
        v_1^{\alpha, \epsilon} \end{bmatrix}}_{\dot{H}^1 \times L^2} \, dt \\
        &=
        \Inner{\int_{-\infty}^\infty \wprop(-t) 
        \begin{bmatrix} 0 \\ F_\mathrm{lin}^{\alpha, \epsilon}(t)
        \end{bmatrix} \, dt}{\begin{bmatrix} v_0^{\alpha, \epsilon} \\
        v_1^{\alpha, \epsilon} \end{bmatrix}}_{\dot{H}^1 \times L^2}.
        \intertext{It follows from \Cref{C:asymp} that agreement of the scattering operators
        implies that}
        \int_{-\infty}^\infty
        \biginner{F_\mathrm{lin}^{\alpha, \epsilon}(t)}{
        u_\mathrm{lin}^{\alpha, \epsilon}(t)}_{L^2} \, dt
        &=
        \int_{-\infty}^\infty
        \biginner{\tilde{F}_\mathrm{lin}^{\alpha, \epsilon}(t)}{
        u_\mathrm{lin}^{\alpha, \epsilon}(t)}_{L^2} \, dt
        \\
        &\hphantom{={}} +
        \bigO{\norm{(u_0^{\alpha, \epsilon}, u_1^{\alpha, \epsilon})}_{\dot{H}^1 \times
        L^2}^9} 
        \cdot
        \norm{(v_0^{\alpha, \epsilon}, v_1^{\alpha, \epsilon})}_{\dot{H}^1 \times L^2} \\
        &=
        \int_{-\infty}^\infty
        \biginner{\tilde{F}_\mathrm{lin}^{\alpha, \epsilon}(t)}{
        u_\mathrm{lin}^{\alpha, \epsilon}(t)}_{L^2} \, dt
        + \bigO[\alpha]{\epsilon^{12}}.
        \addtocounter{equation}{1}\tag{\theequation} \label{F-to-Ftilde}
    \end{align*}

    Now given a $\tau_0 \in \R$, let $\alpha \defeq \frac{1}{2} e^{\tau_0}$ so that $\tau_0
    = \log 2\alpha$. Combining \Cref{L:leb-diff}, \labelcref{F-to-H}, and
    \labelcref{F-to-Ftilde}, we deduce that
    \[
        (H * w)(\tau_0) = (\tilde{H} * w)(\tau_0) + \lilo{1} + \bigO{\epsilon^4}
        \quad\text{as $\epsilon\to 0$}.
    \]
    Taking $\epsilon \to 0$, we arrive at the conclusion.
\end{proof}

\begin{proof}[Proof of \Cref{L:leb-diff}]
    Fix a point $(t_0, x_0) \in \dpt{F}$ and
    let
    \[
        G^{\alpha, \epsilon}(t, x, u) 
        \defeq F(t_0 + (\alpha/\epsilon)^{-2} t, x_0 + (\alpha/\epsilon)^{-2} x, u) u
    \]
    so that
    \begin{align*}
        &\int_{-\infty}^\infty 
        \biginner{F(t, x, u_\mathrm{lin}^{\alpha, \epsilon}(t-t_0, x-x_0))}{
        u_\mathrm{lin}^{\alpha, \epsilon}(t-t_0, x-x_0)}_{L^2_x} \, dt \\
        &\quad=
        (\alpha/\epsilon)^{-8}
        \int_{-\infty}^\infty \int_{\R^3}
        G^{\alpha, \epsilon}(t, x, \alpha u_\mathrm{lin}(t, x)) \, dx \, dt.
    \end{align*}
    Then the conclusion sought can be written as follows: as $\epsilon\to 0$,
    \begin{equation} \label{gae-limit}
        \int_{-\infty}^\infty \int_{\R^3}
        G^{\alpha, \epsilon}(t, x, \alpha u_\mathrm{lin}(t, x))
        - G^{\alpha, \epsilon}(0, 0, \alpha u_\mathrm{lin}(t, x)) \, dx \, dt
        = \lilo[\alpha]{1}.
    \end{equation}

    To prove this, we first recall from the proof of \Cref{L:measures} that
    \[
        u_\mathrm{lin}(t, x)
        \leq
        \begin{cases*}
            2 \cdot 1_{\set{0 < \abs{x} < 2}}(t, x) & if $0 < t < 1$, \\
            \frac{1}{t} \cdot 1_{\set{t-1 \leq \abs{x} < t+1}}(t, x) & if $t > 1$.
        \end{cases*}
    \]
    Hence
    \begin{align*}
        \int_{-\infty}^{\infty} \int_{\R^3} \abs{u_\mathrm{lin}(t, x)}^6 \, dx \, dt
        &= 2 \int_0^\infty \int_{\R^3} |u_\mathrm{lin}(t, x)|^6 \, dx \, dt
        \les 1 + \int_1^\infty \left(\frac{1}{t}\right)^6 t^2 \, dt
        < \infty.
    \end{align*}
    Thus, given any $\eta > 0$, the dominated convergence theorem guarantees that there
    exists an $R > 0$ (depending on $\eta$) so that
    \begin{align}
        &\Abs{
        \iint_{\abs{t} + \abs{x} > R}
        G^{\alpha, \epsilon}(t, x, \alpha u_\mathrm{lin}(t, x)) -
        G^{\alpha, \epsilon}(0, 0, \alpha u_\mathrm{lin}(t, x)) \, dx \, dt} \notag \\
        &\qquad\qquad\les_\alpha
        \iint_{\abs{t} + \abs{x} > R} \abs{u_\mathrm{lin}(t, x)}^6 \, dx \, dt < \eta.
        \label{tx-large}
    \end{align}

    To estimate the integral in \labelcref{gae-limit} over the complementary region $\abs{t}
    + \abs{x} \leq R$, we partition it into the sets
    \[
        U_n^R \defeq \set{(t, x) \in \R\times \R^3 : \abs{t} + \abs{x} \leq R \text{ and }
        \lceil 2\alpha \rceil n/N
        \leq \alpha u_\mathrm{lin}(t, x) 
        < \lceil 2\alpha \rceil (n+1)/N},
    \]
    where $N$ is some large positive integer and $\abs{n} \leq N$.
    For $(t, x) \in U_n^R$, we then have
    \begin{align*}
        \abs{G^{\alpha, \epsilon}(t, x, \alpha u_\mathrm{lin}(t, x)) -
        G^{\alpha, \epsilon}(t, x, \lceil 2\alpha \rceil n/N)} 
        &\les_\alpha 1/N, \\
        \abs{G^{\alpha, \epsilon}(0, 0, \lceil 2\alpha \rceil n/N) - 
        G^{\alpha, \epsilon}(0, 0, \alpha u_\mathrm{lin}(t, x))} 
        &\les_\alpha 1/N,
    \end{align*}
    with implicit constants depending only on $\alpha$.  As $\lceil 2\alpha \rceil
    n/N\in\mathbb{Q}$, replacing the true values of $\alpha u_\mathrm{lin}$ with these
    approximations will allow us to exploit the hypothesis that $(t_0,x_0)$ is a
    determinable point.  To employ these approximations, we first note that
        \begin{align*}
        &\Abs{
        \iint_{\abs{t} + \abs{x} \leq R}
        G^{\alpha, \epsilon}(t, x, \alpha u_\mathrm{lin}(t, x)) -
        G^{\alpha, \epsilon}(0, 0, \alpha u_\mathrm{lin}(t, x)) \, dx \, dt} \\
        &\quad\les_\alpha
        \sum_{\abs{n} \leq N}
        \iint_{U_n^{\mkern-1mu R}}
        \Bigl|G^{\alpha, \epsilon}(t, x, \lceil 2\alpha \rceil n/N) -
        G^{\alpha, \epsilon}(0, 0, \lceil 2\alpha \rceil n/N)\Bigr| \, dx \, dt
        + \frac{R^4}{N}\,.
    \end{align*}
    For each $n$, a change of variables gives
    \begin{align*}
        &\iint_{U_n^R}
        \Bigl|G^{\alpha, \epsilon}(t, x, \lceil 2\alpha \rceil n/N) -
        G^{\alpha, \epsilon}(0, 0, \lceil 2\alpha \rceil n/N)\Bigr| \, dx \, dt \\
        &\ \les_{\alpha,R}
        \epsilon^{-8}\iint_{\abs{t-t_0}+\abs{x-x_0} \leq (\alpha/\epsilon)^{-2} R}
        \Bigl|F(t, x, \lceil 2\alpha \rceil n/N) -
        F(t_0, x_0, \lceil 2\alpha \rceil n/N)\Bigr| \, dx \, dt,
    \end{align*}
    which tends to zero as $\epsilon \to 0$ because $(t_0,x_0)$ is a determinable point.

    Therefore, choosing $N$ sufficiently large (depending on $\eta$) and then $\epsilon$
    sufficiently small (depending on $\eta$), we obtain
    \[
        \Abs{\iint_{\abs{t} + \abs{x} \leq R} 
        G^{\alpha, \epsilon}(t, x, \alpha u_\mathrm{lin}(t, x)) -
        G^{\alpha, \epsilon}(0, 0, \alpha u_\mathrm{lin}(t, x)) \, dx \, dt}
        \les_\alpha \eta.
    \]
    Combining this with \eqref{tx-large} and recalling that $\eta$ was arbitrary, we deduce
    \labelcref{gae-limit}.
\end{proof}

In view of Proposition~\ref{P:conv}, the proof of \Cref{T:main} reduces to showing that the
convolution equation \labelcref{conv-eqn} implies equality of the nonlinearities $F$ and
$\tilde{F}$.  We turn our attention to this task in the next section.

\section{Deconvolutional determination of the nonlinearity}\label{S:deconvolution}

The final step in the proof of \Cref{T:main} consists of formally ``deconvolving'' both
sides of equation~\eqref{conv-eqn} with $w$ to arrive at $H = \tilde{H}$. This in turn
implies that $F(t_0, x_0, {}\cdot{}) = \tilde{F}(t_0, x_0, {}\cdot{})$.  The tool that will
enable us to do so is a Tauberian theorem of Wiener \cite{Wiener}. For the following
formulation of the Tauberian theorem, as well as a very elegant proof, see Korevaar
\cite{Korevaar}.

\begin{theorem}[Wiener's Tauberian theorem] \label{T:wiener}
    Let $f \in L^1(\R)$ and $g \in L^\infty(\R)$. If $f * g = 0$ and $\hat{f}$ has no
    zeroes, then $g = 0$.
\end{theorem}

\begin{proposition} \label{P:w-ft}
    Let $w$ be as defined in \Cref{P:conv}. Then $\hat{w}$ has no zeroes.
\end{proposition}

Assuming that this proposition holds (so that Wiener's Tauberian theorem is
applicable to $w$), \Cref{T:main} follows immediately, as we demonstrate next.

\begin{proof}[Proof of \Cref{T:main}]
    Fix a point $(t_0, x_0) \in \dpt{F} \cap \dpt{\tilde{F}}$ and define $H$ and $\tilde{H}$
    as in \Cref{P:conv} so that $(H - \tilde{H}) * w = 0$. It follows from
    Theorem~\ref{T:wiener} and \Cref{P:w-ft} that $H = \tilde{H}$. In particular,
    \begin{align*} 
        \tfrac{d}{d\tau} \bigl[ e^\tau F(t_0,x_0,e^\tau) - e^\tau 
    \tilde F(t_0,x_0,e^\tau)\bigr] =0, 
    \end{align*} 
    from which it follows that $F(t_0, x_0, {}\cdot{}) = \tilde{F}(t_0, x_0, {}\cdot{})$.
\end{proof}

\begin{proof}[Proof of \Cref{P:w-ft}]
    We decompose $w$ as $w = w_0 + w_1$, where
    \begin{align*}
        w_0(\tau) 
        &\defeq \left( \frac{e^{-3\tau}}{2\vphantom{(e^{-\tau} + 1)^3}} \right) 
        1_{(0, \infty)}(\tau), \\
        w_1(\tau) 
        &\defeq \left( \frac{e^{-3\tau}}{2\vphantom{(e^{-\tau} + 1)^3}}
        -\frac{4e^{-6\tau}}{(e^{-\tau} + 1)^3} \right) 1_{(0, \infty)}(\tau).
    \end{align*}

    First, we compute that
    \begin{equation} \label{w0-expr}
        \hat{w_0}(\xi)
        = \int_0^\infty \frac{e^{-3\tau}}{2} \cdot e^{-i\xi\tau} \, d\tau
        = \frac{1}{6 + 2i\xi}\,.
    \end{equation}
    As $8e^{-3\tau} \leq (e^{-\tau} + 1)^3 \leq 8$ for all $\tau \in (0, \infty)$, 
    we also have $w_1(\tau) \geq 0$ and so
    \[
        \abs{\hat{w_1}(\xi)}
        \leq \int_0^\infty \frac{e^{-3\tau} - e^{-6\tau}}{2} \, d\tau
        = \frac{1}{12}\,.
    \]
    Using the expression \labelcref{w0-expr} for $\hat{w_0}(\xi)$, we find that
    $\abs{\hat{w_0}(\xi)} > \frac{1}{12}$ whenever $\abs{\xi}^2 < 27$, 
    which implies that $\abs{\hat{w}(\xi)} > 0$ for all such $\xi$.

    To handle the remaining $\xi$, we integrate by parts to obtain
    \begin{align*}
        \hat{w_1}(\xi)
        &= \int_0^\infty \frac{d}{d\tau\vphantom{(e^{-\tau} + 1)^3}} 
        \left( \frac{e^{-3\tau}}{2\vphantom{(e^{-\tau} + 1)^3}}
        -\frac{4e^{-6\tau}}{(e^{-\tau} + 1)^3} \right) 
        \frac{e^{-i\xi\tau}}{i\xi\vphantom{(e^{-\tau} + 1)^3}} \, d\tau.
    \end{align*}
    It is straightforward to verify that
    \[
        A(\tau) 
        \defeq - \frac{d}{d\tau\vphantom{(e^{-\tau} + 1)^3}} 
        \left( \frac{e^{-3\tau}}{2\vphantom{(e^{-\tau} + 1)^3}} \right)
        \quad \text{and} \quad
        B(\tau) 
        \defeq \frac{d}{d\tau\vphantom{(e^{-\tau} + 1)^3}} 
        \left( - \frac{4e^{-6\tau}}{(e^{-\tau} + 1)^3} \right) 
    \]
    satisfy $0 \leq \frac{2}{3} B(\tau) \leq A(\tau)$ for all $\tau \in (0, \infty)$. Hence
    \[
        \abs{\hat{w_1}(\xi)}
        \leq \frac{1}{\abs{\xi}} \int_0^\infty \abs{B(\tau) - A(\tau)} \, d\tau
        \leq \frac{1}{\abs{\xi}} \int_0^\infty A(\tau) - \frac{1}{3} B(\tau) \, d\tau
        = \frac{1}{3\abs{\xi}}\,.
    \]
    Using the expression \labelcref{w0-expr} for $\hat{w_0}(\xi)$ again, we find that
    $\abs{\hat{w_0}(\xi)} > \frac{1}{3\abs{\xi}}$ whenever $\abs{\xi}^2 > \frac{36}{5}$, 
    which implies that $\abs{\hat{w}(\xi)} > 0$ for all such $\xi$.
    
    As $\frac{36}5<27$, we conclude that $\hat{w}(\xi) \neq 0$ for all $\xi \in \R$, as was
    to be shown.
\end{proof}

\emergencystretch=1em
\printbibliography

\end{document}